\documentclass[12pt, reqno, a4paper]{amsart}

\usepackage{ amssymb, amsmath, enumerate, amsfonts, amsthm, mathrsfs, url, bm, mathtools}

\usepackage{xcolor}  	
\usepackage[backref]{hyperref}
\hypersetup{
	colorlinks,
    linkcolor={blue!60!black},
    citecolor={blue!60!black},
    urlcolor={red!60!black}
}

\usepackage{color}

\numberwithin{equation}{section}

 \newcommand{\comment}[1]{}  

\usepackage{ amssymb, amsmath, enumerate, amsfonts, amsthm, mathrsfs, url, bm}

\numberwithin{equation}{section}

\usepackage{xcolor}  	
\usepackage[backref]{hyperref}
\hypersetup{
	colorlinks,
    linkcolor={blue!60!black},
    citecolor={blue!60!black},
    urlcolor={red!60!black}
}

\usepackage{color}

\usepackage[margin=1.5in]{geometry}

\RequirePackage{doi}

\usepackage[square,sort,comma,numbers]{natbib}
\newcommand{\C}{\mathbb{C}}

\newcommand{\Z}{\mathbb{Z}}

\newcommand{\R}{\mathbb{R}}

\newcommand{\T}{\mathbb{T}}

\theoremstyle{claim}

 \newcommand{\set}[1]{\left\{#1\right\}}

\newcommand{\abs}[1]{\left| #1\right|}
\newcommand{\bigabs}[1]{\bigl| #1 \bigr|}

\newcommand{\sqbrac}[1]{\left[ #1 \right]}

\newcommand{\ceil}[1]{\left\lceil #1 \right\rceil}

\newcommand{\brac}[1]{\left( #1 \right)}

\newcommand{\Bigbrac}[1]{\Bigl( #1 \Bigr)}
\newcommand{\biggbrac}[1]{\biggl( #1 \biggr)}

\newcommand{\norm}[1]{\left\| #1\right\|}
\newcommand{\bignorm}[1]{\big\| #1 \big\|}

\newcommand{\recip}[1]{\frac{1}{#1}}
\newcommand{\trecip}[1]{\tfrac{1}{#1}}

\newcommand{\intd}{\mathrm{d}}
\newcommand{\supp}{\mathrm{supp}}

\newcommand{\eps}{\varepsilon}

\makeatletter
\let\@@pmod\pmod
\DeclareRobustCommand{\pmod}{\@ifstar\@pmods\@@pmod}
\def\@pmods#1{\mkern4mu({\operator@font mod}\mkern 6mu#1)}
\makeatother

\newcommand{\Spec}{\mathrm{Spec}}

\newtheorem{theorem}{Theorem}[section]

\newtheorem{corollary}[theorem]{Corollary}

\newtheorem{lemma}[theorem]{Lemma}

\theoremstyle{definition}
\newtheorem{definition}[theorem]{Definition}

\numberwithin{theorem}{section}

\renewcommand{\leq}{\leqslant}
\renewcommand{\geq}{\geqslant}

\renewcommand{\ll}{\lesssim}
\renewcommand{\gg}{\gtrsim}

\begin{document}

\title{Solving equations in dense Sidon sets}

\author{Sean Prendiville}
\address{Department of Mathematics and Statistics\\
Lancaster University\\
UK}
\email{s.prendiville@lancaster.ac.uk}




\begin{abstract}
We offer an alternative proof of a  result of Conlon, Fox, Sudakov and Zhao \cite{CFSZRegularity} on solving translation-invariant linear equations in dense Sidon sets. Our proof  generalises to equations in more than five variables and  yields effective bounds.
\end{abstract}

\maketitle

\setcounter{tocdepth}{1}
\tableofcontents
\section{Introduction}

A set $S$ of integers is a \emph{Sidon} set if the only solutions to the equation 
\begin{equation}\label{energy equation}
x-x' = y-y' \quad (x,x', y, y' \in S)
\end{equation} 
are trivial, in the sense that $x=y$ or $x= x'$.  Writing $E(S)$ for the number of tuples $(x,x', y, y')$ solving \eqref{energy equation}, a finite set $S$ is Sidon if and only if $E(S) \leq 2|S|^2-|S|$. 

One can show that if $S\subset [N]$ is Sidon then $|S| \leq (1+o(1))N^{1/2}$, and there are constructions with $|S| \geq (1-o(1))N^{1/2}$, see \cite{OBryantComplete}. Conlon, Fox, Sudakov and Zhao \cite{CFSZRegularity} have shown that Sidon sets whose cardinality is within a constant of this range possess arithmetic structure, in that they contain a solution to any translation-invariant linear equation in five variables, with all variables distinct. Furthermore they are able to demonstrate that this structure is also possessed by \emph{almost Sidon} sets, that is sets for which 
\[
E(S) \leq (2 + o(1))|S|^{2}.
\] 
Their results are deduced using a regularity lemma for graphs with few 4-cycles. We use a Fourier-analytic transference principle developed by Helfgott and de Roton \cite{HelfgottDeRotonImproving} to give an alternative proof of this result, generalising to translation-invariant equations in more variables and extracting  bounds.

\begin{theorem}\label{dense almost sidon}
Let $a_1, \dots, a_s \in \Z\setminus\set{0}$ with $a_1 + \dots + a_s = 0$ and $s \geq 5$.  Given $0 < \delta \leq 1$, suppose that $S \subset [N]$ satisfies
\[
|S| \geq \delta N^{1/2} \quad \text{and} \quad E(S) \leq \brac{2 + \eta} |S|^2.
\] 
Then either 
\[
N \leq \exp\exp(O_{a_i}(1/\delta)), \quad \text{or} \quad \eta \geq \exp\brac{-\exp(O_{a_i}(1/\delta))}
\]
or
\begin{equation}\label{solution lower bound}
\sum_{a_1x_1 + \dots + a_s x_s = 0} \prod_i 1_S(x_i) \geq \exp \brac{-O_{a_i}\brac{1/\delta}} N^{\frac{s}{2} - 1}.
\end{equation}
\end{theorem}

\begin{corollary}\label{dense almost sidon bound}
Let $a_1, \dots, a_s \in \Z\setminus\set{0}$ with $a_1 + \dots + a_s = 0$ and $s \geq 5$.  Given $0 < \delta \leq 1$, suppose that $S \subset [N]$ satisfies
\[
|S| \geq \delta N^{1/2} \quad \text{and} \quad E(S) \leq \brac{2 + \eta} |S|^2.
\] 
If $S $ lacks solutions to the equation 
\begin{equation}\label{main equation}
a_1x_1 + \dots + a_s x_s = 0
\end{equation}
with $x_1, \dots, x_s \in S$ all distinct, then 
\[
N \leq \exp\exp(O_{a_i}(1/\delta)) \quad \text{or} \quad \eta \geq \exp\brac{-\exp(O_{a_i}(1/\delta))}.
\]
\end{corollary}

\begin{corollary}\label{dense sidon bound}
If $S\subset [N]$ is a Sidon set lacking  solutions to \eqref{main equation} with distinct variables then, for $N \geq 3$, we have
\begin{equation}\label{sidon density bound}
|S| = O_{a_i}\brac{ N^{1/2}(\log\log N)^{-1}}.
\end{equation}
\end{corollary}

That such results are obtainable is noted in \cite{CFSZRegularity}, along with a path to proving them. We depart from the use of weak arithmetic regularity suggested therein. Instead our  argument takes advantage of the fact that a Sidon set behaves very nicely with respect to convolution, so that convolving its indicator function with a suitably chosen Bohr set yields a function whose $L^1$ and $L^2$ norms are both comparable to that of a dense set of integers (after appropriate renormalisation). Functions whose $L^p$-norms behave in this manner are similar enough to dense sets of integers for us to import results from the dense setting to  sparse  Sidon sets. This observation originates with Helfgott and de Roton \cite{HelfgottDeRotonImproving}.

An attentive reader will observe that our argument gives a superior exponent of $\log\log N$ than that stated in Corollary \ref{dense sidon bound}. Furthermore the improved exponent grows as the number of variables in \eqref{main equation} increases. This is due to our use of a result of Bloom \cite{BloomTranslation} which counts the number of solutions to a translation-invariant equation in a dense set of integers. For equations in four or more variables, there is a more effective density bound due to Schoen and Sisask \cite{SchoenSisaskRoth}. As is indicated in Sanders' \emph{Mathematical Review}\footnote{\texttt{MR3482282}.} of this paper, one may adapt the argument\footnote{Replace the sum set $A+A$ with a suitable set of popular sums. Thanks to Thomas Bloom and Olof Sisask for pointing this out.} to improve Bloom's counting result.  This then  improves \eqref{sidon density bound} to 
\[
|S| = O_{a_i}\brac{ \frac{N^{1/2}}{\exp\brac{\brac{\log\log N}^{\Omega(1)}}}}.
\] 

The author would be very interested in any proof which yields a polylogarithmic bound in Corollary \ref{dense sidon bound}. For dense sets of integers, all polylogarithmic bounds require some kind of localisation from the interval $[N]$ to a sparser substructure, such as a subprogression or Bohr set.  When dealing with sparse sets of integers like Sidon sets, such localisation is lossy, because the sparse set can be even sparser on the substructure. An example to bear in mind is that a subset of $[N]$ of cardinality $\sqrt{N}$ may intersect each subinterval of length $\sqrt{N}$ in at most one point. The author believes that obtaining a polylogarithmic bound in Corollary \ref{dense sidon bound} may be a  model problem for improving bounds in Roth's theorem in the primes \cite{GreenRoth, HelfgottDeRotonImproving, NaslundImproving}.

\subsection*{Paper Organisation}
We prove Theorem \ref{dense almost sidon} in \S\ref{transference section}, assuming three key lemmas. Proving these lemmas occupies \S\S\ref{bloom section}--\ref{dense model section}. We deduce Corollaries \ref{dense almost sidon bound} and \ref{dense sidon bound} in \S\ref{corollaries section}
\subsection*{Acknowledgements}
The author thanks Jonathan Chapman for corrections, Sam Chow for numerous useful conversations, and Yufei Zhao for an inspiring talk in the (online) Stanford Combinatorics Seminar.
\subsection*{Notation}

\subsubsection*{Standard conventions}
We use $[N]$ to denote the interval of integers $ \{ 1,2, \dots, N\}$.  We use counting measure on $\Z$, so that for $f,g :\Z \to \C$, we have
$$
\norm{f}_{p} := \biggbrac{\sum_x |f(x)|^p}^{\recip{p}}\ \text{and}\ (f*g)(x) := \sum_y f(y)g(x-y).
$$ 
Any sum of the form $\sum_x$ is to be interpreted as a sum over $\Z$. The \emph{support} of $f$ is the set $\supp(f) := \set{x \in \Z : f(x) \neq 0}$. 
 
We use Haar probability measure on $\T := \R/\Z$, so that for integrable $F : \T \to \C$, we have
\begin{equation*}
\norm{F}_{p} := \biggbrac{\int_\T |F(\alpha)|^p\intd\alpha}^{\recip{p}} = \biggbrac{\int_0^1 |F(\alpha)|^p\intd\alpha}^{\recip{p}}
\end{equation*}
and
\begin{equation*}
\norm{F}_{\infty} := \sup_{\alpha \in \T} |F(\alpha)|^.
\end{equation*}
Write $\norm{\alpha}_\T$ for the distance from $\alpha \in \R$ to the nearest integer
$
\min_{n \in \Z} |\alpha - n|.
$
This remains well-defined on $\T$.

\begin{definition}[Fourier transform]\label{fourier transform}
For $f : \Z \to \C$ with finite support define $\hat{f} : \T \to \C$ by
$$
\hat{f}(\alpha) := \sum_{n\in \Z} f(n) e(\alpha  n).
$$
Here $e(\beta)$ stands for $e^{2\pi i \beta}$.
\end{definition}

\subsubsection*{Asymptotic notation}
For a complex-valued function $f$ and positive-valued function $g$, write $f \ll g$ or $f = O(g)$ if there exists a constant $C$ such that $|f(x)| \le C g(x)$ for all $x$. We write $f = \Omega(g)$ if $f \gg g$.  The notation $f\asymp g$ means that $f\ll g$ and $f\gg g$. We subscript these symbols if the implicit constant depends on additional parameters.

 We write $f = o(g)$ if for any $\eps > 0$ there exists $X\in \R$ such that for all $x \geq X$ we have $|f(x)| \leq \eps g(x)$.

\subsubsection*{Local conventions}
As indicated in the introduction, we define the additive energy of a finitely supported function $f : \Z \to \R$ to be the quantity
\[
E(f) := \sum_{x-x'= y-y'} f(x)f(x')f(y)f(y').
\]
When $f =1_S$ is the characteristic function of a finite set $S \subset \Z$ we write $E(S)$.  Notice that 
\[
E(S) = \sum_n r_S(n)^2
\]
where 
\[
r_S(n) := \sum_{n_1 - n_2 = n} 1_S(n_1)1_S(n_2)
\]
is the number of representation of $n$ as a difference of elements of $S$. In the literature this notation is sometimes used for the number of representations as a \emph{sum} of two elements of $S$.

\section{The transference argument}\label{transference section}

In this section we prove Theorem \ref{dense almost sidon} assuming the following three ingredients.

\begin{lemma}[$L^2$--Bloom]\label{l2 bloom}
Let $a_1,\dots,a_s\in \Z \setminus\set{0}$ with $s \geq 5$ and $a_1 + \dots + a_s = 0$. Let $f:  I \to [0, \infty)$ be a function defined on an interval $I \subset \Z$ of length $N$. If $\sum_n f(n) \geq \delta N$ and  $\sum_n f(n)^2 \leq N$ then we have the lower bound
\[
\sum_{a_1x_1 + \dots + a_sx_s = 0}f(x_1)\dotsm f(x_s) \geq  \exp\brac{-O_{a_i}\brac{1/\delta}}N^{s-1}.
\]
\end{lemma}

We deduce this from a theorem of Bloom \cite{BloomTranslation} in \S\ref{bloom section}.

\begin{lemma}[Counting lemma for bounded energy functions]\label{sidon counting lemma}
Let $s \geq 5$ and  $a_1, \dots, a_s \in \Z\setminus\set{0}$.  Let $\nu:  I \to [0, \infty)$ be a function defined on an interval $I \subset \Z$ of length $N$.  Suppose that 
\[
 \sum_n \nu(n) \leq N \quad \text{and} \quad E(\nu) \leq N^3.
\]
Then for any  $|f_i| \leq \nu $ we have
\[
\abs{\sum_{a_1x_1 + \dots + a_sx_s = 0} f_1(x_1) \dotsm f_s(x_s)} \leq N^{s-1}  \frac{\min_i\bignorm{\hat f_i}_{\infty}}{\bignorm{\hat1_{[N]}}_{\infty}}.
\]
\end{lemma}

This is proved in \S\ref{counting section}.

\begin{lemma}[Dense model for almost-Sidon sets]\label{almost sidon model}
Let $0 \leq \eta \leq1$ and suppose that $S \subset [N]$ satisfies
\[
|S| \geq \delta N^{1/2} \quad \text{and} \quad E(S) \leq \brac{2 +\eta}|S|^2.
\]  
Then for any $0 < \eps \leq \min\set{\trecip{2}, \delta}$ there exists $f: (-\eps N, (1+\eps)N] \to [0, \infty)$ such that all of the following hold 
\begin{itemize}
\item $\displaystyle \sum_n f(n) = N^{1/2}|S|$;
\item $\displaystyle  \bignorm{\hat{f}-N^{1/2}\hat{1}_S}_\infty \leq \eps N$;
\item $\displaystyle \sum_n f(n)^2 \leq N\sqbrac{1 + \brac{\eta  + N^{-1/2}}(1-\eta)^{-1}\exp\brac{\eps^{-O(1)}}}$.
\end{itemize}
\end{lemma}

The above constitutes the main idea in our approach and is proved in \S\ref{dense model section}.

\begin{proof}[Proof of Theorem \ref{dense almost sidon}]
We may assume that $\eta \leq 1/2$, for the second possible conclusion of the theorem is that $\eta$ is large. Let us apply Lemma \ref{almost sidon model},  with $\eps$ to be chosen.  This gives $f : (-\eps N, (1+\eps)N] \to [0, \infty)$ satisfying $\sum_n f(n) \geq \delta N$, $\bignorm{\hat{f}-N^{1/2}\hat{1}_S}_\infty \leq \eps  N$ and 
\begin{equation}\label{f l2 bound}
\sum_n f(n)^2 \leq N\sqbrac{1 + \brac{ \eta  +  N^{-1/2}} \exp\brac{\eps^{-O(1)}}}.
\end{equation}
By \eqref{f l2 bound}, either $\sum_n f(n)^2 \leq 2N$ or one of the following two possibilities holds
\begin{equation}\label{eta N bounds}
 \eta \geq \exp\brac{-\eps^{-O(1)}}\quad \text{or} \quad N \leq \exp\brac{\eps^{-O(1)}}.
\end{equation}
Notice that $(-\eps N, (1+\eps)N]$ is an interval of length at most $2N$.
 Hence, assuming that neither option in \eqref{eta N bounds} holds,  Lemma \ref{l2 bloom} gives that
\[
\sum_{a_1x_1 + \dots + a_sx_s = 0}f(x_1)\dotsm f(x_s) \geq \exp\brac{-O_{a_i}\brac{1/\delta}}N^{s-1}.
\]

Define $\nu := f + N^{1/2}1_S$.  We claim that, provided we divide through by a suitable absolute constant, the function $\nu$ satisfies the hypotheses of Lemma \ref{sidon counting lemma} on the interval $I = (-\eps N, (1+\eps)N]$. By the triangle inequality in $L^4$, and the Fourier-analytic interpretation of energy, we have
\begin{multline*}
E(\nu)^{1/4} = \norm{\hat{\nu}}_4 \leq \bignorm{\hat{f}}_4 + N^{1/2}\norm{\hat1_S}_4 \leq \norm{f}_1^{1/2}\bignorm{\hat{f}}^{1/2}_2 + N^{1/2} E(S)^{1/4}\\ \ll N^{1/4}\norm{f}_2+ N^{1/2} N^{1/4} \ll N^{3/4}.
\end{multline*}
Assuming that neither option in \eqref{eta N bounds} holds, we compare Fourier coefficients at zero to deduce that
\begin{multline*}
\sum_n \nu(n) =  \hat{f}(0)+N^{1/2}\hat{1}_{S}(0) 
\leq 2 \hat{f}(0) + \eps N\\ \leq 2(2N)^{1/2} \Bigbrac{\sum_n f(n)^2}^{1/2} + \eps N \ll N.
\end{multline*}

We may therefore apply Lemma \ref{sidon counting lemma} together with a telescoping identity to deduce that
\[
\abs{\sum_{a_1x_1 + \dots + a_sx_s = 0} \brac{\prod_i f(x_i) - \prod_i N^{1/2}1_S(x_i)}} \ll_s \eps  N^{s-1}.
\]
Hence either we deduce \eqref{solution lower bound}, or one of the following holds
\begin{itemize}
\item $\displaystyle
\eps \geq \exp\brac{-O_{a_i}\brac{1/\delta}}$;
\item $\displaystyle \eta \geq \exp\brac{-\eps^{-O(1)}}$;
\item $\displaystyle N \leq \exp\brac{\eps^{-O(1)}}$.
\end{itemize}
We obtain the result on taking $\eps$ sufficiently small to preclude the first possibility.
\end{proof}

\section{Results on dense sets of integers}\label{bloom section}

\begin{theorem}[Bloom \cite{BloomTranslation}]\label{bloom}
Let $a_1,\dots,a_s\in \Z \setminus\set{0}$ with $s \geq 5$ and $a_1 + \dots + a_s = 0$. Then for any $A \subset [N]$ with $|A|\geq \delta N$ we have the lower bound
\[
\sum_{a_1x_1 + \dots + a_sx_s = 0}1_A(x_1)\dotsm 1_A(x_s)\geq  \exp\brac{-O_{a_i,\eps}\brac{\delta^{-\recip{s-2} - \eps}}}N^{s-1}.
\]
\end{theorem}

\begin{proof}[Proof of Lemma \ref{l2 bloom}]
Translating, we may assume that $I = [N]$. Define
$$
A := \set{x \in [N] : f(x) \geq \delta/2}.
$$
Then, employing the Cauchy--Schwarz inequality, we have
\begin{equation}\label{cauchy employment}
\begin{split}
\delta N \leq \sum_x f(x)  &= \sum_{x \notin A} f(x) + \sum_{x \in A} f(x)\\
&\leq \trecip{2}\delta N  + |A|^{1/2}\brac{\sum_x  f(x)^2}^{1/2}\\
& \leq \trecip{2}\delta N + \brac{|A|N}^{1/2}.
\end{split}
\end{equation}
Therefore
\begin{equation}\label{large value density}
|A| \geq \tfrac{\delta^2}{4} N.
\end{equation}

Applying Theorem \ref{bloom} we deduce that
\begin{align*}
\sum_{a_1x_1 + \dots + a_sx_s = 0} f(x_1) \dotsm f(x_s) & \geq (\delta/4)^s\sum_{a_1x_1 + \dots + a_sx_s = 0} 1_A(x_1) \dotsm 1_A(x_s)\\ & \gg \delta^s \exp\brac{-O_{a_i}\brac{\delta^{-0.8}}}N^{s-1}.\qedhere
\end{align*}
\end{proof}

\section{An almost-Sidon counting lemma}\label{counting section}

\begin{proof}[Proof of Lemma \ref{sidon counting lemma}]
For any finitely supported $f : \Z \to \C$ and $a \in \Z\setminus \set{0}$ we have
\begin{equation*}
\int_\T \bigabs{\hat{f}(a\alpha)}^{s-1} \intd\alpha \leq \Bigbrac{\sum_n |f(n)|}^{s-5} \int_\T \bigabs{\hat{f}(a\alpha)}^{4} \intd\alpha.
\end{equation*}
If $|f| \leq \nu$, then
\[
\sum_n |f(n)| \leq \sum_n \nu(n) \leq N.
\]
By orthogonality
\begin{multline*}
\int_\T \bigabs{\hat{f}(a\alpha)}^{4} \intd\alpha = \sum_{x-x' = y-y'} f(x)\overline{f(x')f(y)}f(y') \\
\leq \sum_{x-x' = y-y'} \nu(x)\nu(x')\nu(y)\nu(y') \leq N^3.
\end{multline*}
Therefore 
\[
\int_\T \bigabs{\hat{f}(a\alpha)}^{s-1} \intd\alpha \leq N^{s-2} .
\]

Again by orthogonality, together with H\"older's inequality
\begin{multline*}
\abs{\sum_{a_1x_1 + \dots + a_sx_s = 0} f_1(x_1) \dotsm f_s(x_s)} = \abs{\int_\T \hat{f}_1(a_1\alpha) \dotsm \hat{f}_s(a_s\alpha)\intd\alpha}\\ \leq  \bignorm{\hat{f}_i}_{\infty} \prod_{j \neq i} \brac{\int_\T \bigabs{\hat{f}(a\alpha)}^{s-1} \intd\alpha}^{\recip{s-1}} \leq \bignorm{\hat{f}_i}_{\infty} N^{s-2}.\qedhere
\end{multline*}
\end{proof}

\section{A modelling lemma for almost-Sidon sets}\label{dense model section}
We begin our proof of Lemma \ref{almost sidon model} with two subsidiary results on almost Sidon sets.
\begin{lemma}\label{large rep}
Let $S \subset [N]$ satisfy
\[
E(S):= \sum_{x-x' = y-y'}1_S(x)1_S(x')1_S(y)1_S(y') \leq (2 +\eta)|S|^2.
\]  
Then, on writing
\[
r_S(n) := \sum_{n_1 - n_2 = n} 1_S(n_1)1_S(n_2),
\]
we have
\[
\sum_{\substack{r_S(n) > 1\\ n \neq 0}} r_S(n) \leq \eta |S|^2 + |S|.  
\]
\end{lemma}

\begin{proof}
We observe that
\begin{multline*}
\sum_{\substack{r_S(n) > 1\\ n\neq 0}} r_S(n) \leq \sum_{n\neq 0} r_S(n)(r_S(n) -1) = \sum_{n\neq 0} r_S(n)^2 - \sum_{n\neq 0} r_S(n) \\
\leq (1+\eta)|S|^2 - \brac{|S|^2-|S|} =\eta |S|^2 + |S|.\qedhere
\end{multline*}
\end{proof}

\begin{lemma}\label{almost sidon upper bound}
Let $\eta \in [0,1)$ and suppose that $S \subset [N]$ satisfies
\[
E(S):= \sum_{x-x' = y-y'}1_S(x)1_S(x')1_S(y)1_S(y') \leq (2 +\eta)|S|^2.
\]  
Then
\[
|S| \leq 2\brac{\tfrac{N}{1-\eta}}^{1/2}.
\]
\end{lemma}

\begin{proof}
Using Lemma \ref{large rep} we have
\[
|S|^2 = \sum_{\substack{r_S(n) \leq 1\\ n \neq 0}} r_S(n)+ \sum_{\substack{r_S(n) > 1\\ n \neq 0}} r_S(n) + |S| \leq 2N + \eta|S|^2 + 2|S|.\qedhere
\]
\end{proof}

We are now in a position to prove Lemma \ref{almost sidon model} in earnest.

\begin{proof}[Proof of Lemma \ref{almost sidon model}]
Define the large spectrum of $S$ to be the set 
$$
\Spec(S, \eps ) := \set{\alpha \in \T : |\hat{1}_S(\alpha)| \geq \eps |S|}.
$$
Define the Bohr set 
\begin{equation}\label{bohr set defn}
B := \set{n \in [-\eps N, \eps N] : \norm{n\alpha}_\T \leq \eps\quad \forall \alpha \in \Spec(S, \eps )}.
\end{equation}
Write $\mu_B$ for the normalised characteristic function of $B$, so that
$$
\mu_B := |B|^{-1} 1_{B}.
$$
Then we define
\begin{equation}\label{g defn}
f :=  N^{1/2}1_S * \mu_B ,
\end{equation}
where, for finitely supported $f_i$, we set
$$
f_1 * f_2(n) := \sum_{m_1+m_2 = n} f_1(m_1) f_2(m_2).
$$

It is straightforward to check that $f$ is supported on $(-\eps N, (1+\eps)N]$ and that $\sum_n f(n) = N^{1/2}|S|$.  Let us next estimate $|N^{1/2}\hat{1}_S- \hat{f}|$.  The key identity is 
\[
\widehat{f_1*f_2} = \hat{f}_1\hat{f}_2,
\]
so that $|N^{1/2}\hat{1}_S- \hat{f}| = N^{1/2}|\hat{1}_S||1-\hat\mu_B|$.

If $\alpha \notin \Spec(S, \eps )$ then we have
$$
|N^{1/2}\hat{1}_S(\alpha)- \hat{f}(\alpha)| = N^{1/2}|\hat{1}_S(\alpha)||1 - \hat{\mu}_B(\alpha)| \leq 2N^{1/2} \eps |S|.
$$
If $\alpha \in \Spec(S, \eps )$, then for each $n \in B$ we have
$
e(\alpha n) = 1 + O(\eps).
$  
Hence
$
\hat{\mu}_B(\alpha) = 1 + O(\eps),
$
and consequently
$$
|N^{1/2}\hat{1}_S(\alpha)- \hat{f}(\alpha)| = N^{1/2}|\hat{1}_S(\alpha)||1 - \hat{\mu}_B(\alpha)| \ll N^{1/2} |S|\eps.
$$
Combining both cases and Lemma \ref{almost sidon upper bound} gives 
\begin{equation}\label{fourierestimate}
\bignorm{N^{1/2}\hat{1}_S - \hat{f}}_\infty \ll \eps N.
\end{equation}

We have 
\begin{equation}\label{hoped for L2}
\sum_n f(n)^2  = N|B|^{-2}\sum_{n_1 - n_2 = m_1 - m_2} 1_S(n_1)1_S(n_2)1_B(m_1)1_B(m_2) .
\end{equation}
Write
\[
r_S(n) := \sum_{n_1 - n_2 = n} 1_S(n_1)1_S(n_2).
\]
Then by Lemma \ref{large rep}, the inner sum in  \eqref{hoped for L2} is
\begin{multline*}
 \sum_n r_S(n)r_B(n) \leq |B|^2 + |B|\sum_{\substack{r_S(n) > 1\\ n \neq 0} } r_S(n)+ |B||S|\\
\leq |B|^2 + \brac{\eta |S|^2+2|S|}|B|. 
\end{multline*}
Using the estimate $|S| \ll (1-\eta)^{-1/2} N^{1/2}$ afforded by Lemma \ref{almost sidon upper bound}, it remains to establish the lower bound
\begin{equation}\label{bohr lower bound}
|B| \geq \exp\brac{-\eps^{O(1)}} N.
\end{equation}

Let $\alpha_1, \dots, \alpha_R$ be a maximal $(1/N)$-separated subset of $\Spec(S, \eps)$. Since every element of $\Spec(S, \eps)$ is within $1/N$ of some $\alpha_i$, one can check that
\begin{equation}\label{bohr inclusion}
B \supset \set{n \in [-\eps N/2, \eps N/2] : \norm{\alpha_i n}_\T \leq \eps/2\quad \forall i = 1,\dots, R}.
\end{equation}
Hence the argument proving the standard lower bound for  Bohr sets (e.g.\ \cite[Lemma 4.2]{TaoVuAdditive}) gives
\[
|B| \geq  \ceil{4/\eps}^{1+R} N.
\]

By the large sieve inequality (e.g.\ \cite[Lemma 5.3]{VaughanHardy}) we have
\[
R \eps^4 |S|^4 \leq \sum_{i=1}^R \bigabs{\hat{1}_S(\alpha_i)}^4 \ll N \sum_n r_S(n)^2 \ll N \brac{2+\eta} |S|^2.
\]
Hence  $R \ll  \delta^{-2}\eps^{-4}$. \end{proof}

\section{Proof of corollaries}\label{corollaries section}

\begin{proof}[Proof of Corollary \ref{dense almost sidon bound}]
Let us first obtain an upper bound for the number of solutions in $S$ to the equation
\begin{equation*}\label{main equation 2}
a_1x_1 + \dots + a_s x_s = 0.
\end{equation*}
 By our hypotheses, all such solutions should have $x_i = x_j$ for some $i \neq j$.  At the cost of a factor of $\binom{s}{2}$, we may assume that $x_{s-1}=x_s$. Writing $n:= a_{4}x_{4} +\dots + a_sx_s$, the number of choices for the remaining three variables is at most 
 \begin{multline*}
 \sum_{a_1x_1 + a_2x_2 + a_{3}x_{3} = -n} 1_S(x_1)1_S(x_2)1_S(x_3)  = \int_\T e(\alpha n)\prod_{i=1}^3\hat{1}_S(a_ix_i)  \intd\alpha\\ \leq \prod_{i=1}^{3} \brac{\int_\T \bigabs{\hat{1}_S(a_i\alpha_i)}^{3}\intd\alpha}^{\frac{1}{3}} \leq \prod_{i=1}^{3} \brac{\int_\T \bigabs{\hat{1}_S(a_i\alpha_i)}^{4}\intd\alpha}^{\frac{1}{4}}
 = E(S)^{\frac{3}{4}}.
 \end{multline*}
 We may assume that $\eta \leq 1/2$ (otherwise we are done). Using Lemma \ref{almost sidon upper bound} we deduce that the number of choices for $x_1, x_2, x_3$ is $O(N^{3/4})$. Since there are $N^{\frac{s-4}{2}}$ choices for the remaining variables, we deduce that
 \[
 \sum_{a_1x_1 + \dots + a_s x_s = 0} \prod_i 1_S(x_i) \ll_s  N^{\frac{s}{2} -  \frac{5}{4}}.
 \]
 We deduce the result on comparing this with the lower bound given in Theorem \ref{dense almost sidon}
 \end{proof}
 
 We obtain Corollary \ref{dense sidon bound} from Corollary \ref{dense almost sidon bound} since, for a Sidon set $S$, we have $E(S) \leq (2+\eta)|S|$ with $\eta = 0$.

  \bibliographystyle{alphaabbr}
 \bibliography{/Users/sean/Dropbox/SeanBib} 

\end{document}